\documentclass[multi]{cambridge7A}
\usepackage[english,UKenglish]{babel}
\usepackage[longnamesfirst,sectionbib]{natbib}
\usepackage{amsmath,amssymb,amsthm,amsfonts}
\usepackage{graphicx,epsfig,psfrag}
\usepackage{color,calc,xspace,latexsym,ifthen,enumerate,url}

\setcitestyle{numbers,square,comma}

\theoremstyle{plain}
\newtheorem{theorem}{Theorem}[section]
\newtheorem{lemma}[theorem]{Lemma}

\theoremstyle{remark}
\newtheorem*{note}{Note}

\allowdisplaybreaks[1]

 \newcommand{\nc}{\newcommand}
 \newcommand{\rnc}{\renewcommand}
 
\nc{\beqn}[1]{\begin{equation}\label{#1}} \nc{\ee}{\end{equation}}
 \rnc{\l}{\left} \rnc{\r}{\right}
 \nc{\ri}{\mathrm{i}}  \nc{\rd}{\mathrm{d}}
 \nc{\SL}{\mathrm{SL}}\nc{\PSL}{\mathrm{PSL}}
 \nc{\re}{\mathrm{e}}
 \nc{\f}{\varphi} \nc{\g}{\gamma} \nc{\kp}{\kappa}
 \nc{\lm}{\lambda} \nc{\om}{\omega} \nc{\s}{\sigma}
 \nc{\vt}{\vartheta} \rnc{\a}{\alpha} \rnc{\b}{\beta}
 \rnc{\d}{\delta} \rnc{\t}{\theta} \nc{\e}{\varepsilon}
 \nc{\G}{\Gamma}
 \nc{\oQ}{\mathbb{Q}} \nc{\oR}{\mathbb{R}} \nc{\oC}{\mathbb{C}}
 \nc{\oZ}{\mathbb{Z}} \nc{\oA}{\mathbb{A}} \nc{\oN}{\mathbb{N}}
 \nc{\oH}{\mathbb{H}} \nc{\oE}{\mathbb{E}} \nc{\oP}{\mathbb{P}}
 \nc{\sO}{\mathcal{O}}\nc{\sI}{\mathcal{I}}\nc{\sH}{\mathcal{H}}
 \nc{\sF}{\mathcal{F}}\nc{\sA}{\mathcal{A}}\nc{\sQ}{\mathcal{Q}}
 \nc{\sL}{\mathcal{L}}
 \nc{\ga}{\mathfrak{a}}  \nc{\gb}{\mathfrak{b}}
 \nc{\gc}{\mathfrak{c}}
 \nc{\gw}{\mathfrak{w}}
  \nc{\gp}{\mathfrak{p}} \nc{\gq}{\mathfrak{q}}
 \nc{\gA}{\mathfrak{A}} \nc{\gB}{\mathfrak{B}}
\nc{\bv}{\mathbf{v}}\nc{\bV}{\mathbf{V}}

 \nc{\abar}{\overline{\a}} \nc{\bbar}{\overline{\b}}
 \nc{\gbar}{\overline{\g}} \nc{\dbar}{\overline{\d}}
 \nc{\taubar}{\overline{\tau}}\nc{\lmbar}{\overline{\lm}}
 \nc{\kpbar}{\overline{\kp}}\nc{\mubar}{\overline{\mu}}
 \nc{\tbar}{\overline{\t}}
\nc{\ws}{\widetilde{\s}} \nc{\wPhi}{\widetilde{\Phi}}
\nc{\wX}{\widetilde{X}}\nc{\wQ}{\widetilde{Q}}
\nc{\wsL}{\widetilde{\sL}}
 \nc{\Eq}{\ =\ }\nc{\Ge}{\ \ge\ }\nc{\Le}{\ \le\ }
 \nc{\Def}{\ :=\  }
  \rnc{\mod}{\ \mathrm{mod}\ }
 \providecommand{\dummy}{} \nc{\sfrac}[2]{
  \renewcommand{\dummy}{\scriptstyle{\frac{#1}{#2}}}
  {\raise 0.5pt\hbox{$\dummy$}}\displaystyle}
 \nc{\half}{\sfrac{1}{2}}
 \nc{\norm}{\mathrm{norm}} \nc{\ClGp}{\mathrm{ClGp}}
 \nc{\Nm}{\mathrm{Nm}}
 \nc{\idl}[1]{\ensuremath{\big\langle#1\big\rangle}}
 \rnc{\Im}{\mathrm{Im}} \rnc{\Re}{\mathrm{Re}}
 \nc{\mat}[1]{\begin{pmatrix}#1\end{pmatrix}}
 \nc{\inter}{\cap} \nc{\less}{\smallsetminus}
 \nc{\dby}[2]{\frac{d#1}{d#2}} \nc{\ddby}[2]{\frac{d^2#1}{d#2^2}}
 \nc\del{\partial} \nc{\delby}[2]{\frac {\del {#1}}{\del{#2}} }
\nc{\bdelby}[2]{\displaystyle\frac {\del {#1}}{\del{#2}} }
\nc{\ddelby}[2]{\frac{\del^2 {#1}}{\del {#2} ^2} }
\nc{\One}{\mathbf{1}}
\nc{\ufrac}[2]{\scriptstyle{\frac{#1}{#2}}}
\nc{\uhalf}{\ufrac{1}{2}}

\begin{document}
\alphafootnotes
\author[David Williams]{David Williams\footnotemark }
\chapter[Dynamical-system picture of a phase transition]%
{A dynamical-system picture of a simple branching-process phase
transition}
\footnotetext[1]{Mathematics Department, Swansea University, Singleton Park,
  Swansea SA2 8PP; dw@reynoldston.com}
\arabicfootnotes
\contributor{David Williams \affiliation{Swansea University}}
\renewcommand\thesection{\arabic{section}}
\numberwithin{equation}{section}
\renewcommand\theequation{\thesection.\arabic{equation}}
\numberwithin{figure}{section}
\renewcommand\thefigure{\thesection.\arabic{figure}}

\begin{abstract}
This paper proves certain results from the `appetizer for
non-linear Wiener--Hopf theory', \cite{WNL}. Like that paper, it
considers only the simplest possible case in which the underlying
Markov process is a two-state Markov chain.  Key generating
functions provide solutions of a simple two-dimensional dynamical
system, and the main interest is in the way in which Probability
Theory and ODE theory complement each other.  No knowledge of
either ODE theory or Wiener--Hopf theory is assumed. Theorem 1.1
describes one aspect of a phase transition which is more
strikingly conveyed by Figures \ref{Super} and \ref{crit}.
\end{abstract}

\subparagraph{AMS subject classification (MSC2010)}60J80, 34A34

\section{Introduction}

This paper is a development of something I mentioned briefly in
talks I gave at Bristol, when John Kingman was in the audience,
and at the Waves conference in honour of John Toland at Bath. I
thanked both John K and John T for splendid mathematics and for
their wisdom and kindness.

The main point of the paper is to prove Theorem \ref{phase} and
related results in a way which emphasizes connections with a
simple dynamical system. The phase transition between Figures
\ref{Super} and \ref{crit} looks more dramatic than the famous
$1$-dimensional result we teach to all students.

The model studied here is a special case of the model introduced
in \cite{WNL}.  I called that paper, which contained no proofs, an
`appetizer'; but before writing a fuller version, I became caught
up in Jonathan Warren's enthusiasm for the relevance of
\emph{complex} dynamical systems (in $\oC^2$). See \citet*{WW}.
This present paper, \emph{completely independent of the earlier
appetizer} and of my paper with Warren, can, I hope, provide a
more tempting appetizer for what I called `non-linear Wiener--Hopf
theory'.  No knowledge of any kind of Wiener--Hopf theory is
assumed here.

I hope that Simon Harris and I can throw further light on the
models considered here, on the other models in \cite{WNL}, and on
still other, quite different, models.

\subparagraph{Our model.}A particle moving on the real line can either be
of type $+$ in which case it moves right at speed $1$ or of type
$-$ in which case it moves left at speed $1$.

Let $q_-$ and $q_+$ be fixed numbers with $q_- > q_+ > 0$, and let
$\b$ be a positive parameter. We write $K_{\pm}=q_{\pm}+\b$. So,
to display things, we have
\beqn{Assms}
    q_- > q_+ > 0, \quad \b > 0, \quad K_+ = q_+ + \b,
    \quad K_- = q_- + \b.
\ee
We define
\[
   \b_{\rm c} := {\textstyle \frac12}\l(\sqrt{q_-} - \sqrt{q_+}\r)^2.
\]

A particle of type $\pm$ can flip to the `opposite' type at rate
$q_{\pm}$ and can, at rate $\b$, die and at its death give birth
to two daughter particles (of the same type and position as their
`parent'). This is why $\b$ is a `birth rate'. The usual
independence conditions hold.

\begin{theorem}\label{phase}
Suppose that our process starts at time $0$ with just $1$ particle
of type $+$ at position $0$.
\begin{enumerate}[{\rm (a)}]
\item Suppose that $\b>\b_{\rm c}$.  Then, with probability $1$,
each of infinitely many particles will spend time to the left of
$0$.

\item Suppose instead that $\b\le \b_{\rm c}$.  Then, with
probability not less than $1~-~\sqrt{q_+/q_-}$, there will never
be any particles to the left of $0$.
\end{enumerate}
\end{theorem}

Large-deviation theory (of which the only bit we need is proved
here) allows one to prove easily that if $\b<\b_{\rm c}$, then,
almost surely, only a finite number of particles are ever to the
left of $0$.

The interplay between the Probability and the ODE theory is what
is most interesting.  We shall see that $\b_{\rm c}$ plays the
r\^ole of a critical parameter in several ways, some
probabilistic, some geometric.  The `balance' which occurs when
$\b=\b_{\rm c}$ is rather remarkable.

The paper poses a tantalizing problem which I cannot yet solve.

\section{Wiener--Hopferization}
\subsection{The processes $\{N^{\pm}(\f): \f\ge 0\}$}
For any particle $i$ alive at time $t$, we define $\Phi_i(t)$ to
be its position on the real line at time $t$, and we extend the
definition of $\Phi_i$ by saying that at any time $s$ before that
particle's birth, $\Phi_i(s)$ is the position of its unique
ancestor alive at time $s$.

So far, so sane!  But we are now going to Wiener--Hopferize
everything with a rather clumsy definition which defines for each
$\f\ge 0$ two subsets, $S^+(\f,\b)$ and $S^-(\f,\b)$, of
particles.

We put particle $i$ in set $S^+(\f,\b)$ if there is some $t$ in
$[B(i),D(i))$ where $B(i)$ and $D(i)$ are, respectively, the times
of birth and death of particle $i$, such that
\begin{itemize}
\item $\Phi_i(t)=\f$,
\item $\Phi_i(t)\ge \max\{\Phi_i(s):s\le t\}$, and
\item $\Phi_i$ grows to the right of $t$ in that, for $\e>0$, there
exists a $\d$ with $0<\d<\e$ such that $\Phi_i(\cdot)>\f$
throughout $(t,t+\d)$.
\end{itemize}

At the risk of labouring things, let me describe $S^-(\f,\b)$ for
$\f\ge 0$. We put particle $i$ in set $S^-(\f,\b)$ if there is
some $t$ in $[B(i),D(i))$ such that $\Phi_i(t)=-\f$,
$\Phi_i(t)\le \min\{\Phi_i(s):s\le t\}$, and $\Phi_i$ decreases to
the right of $t$ in that, for $\e>0$, there exists a $\d$ with
$0<\d<\e$ such that $\Phi_i(\cdot)<-\f$ throughout $(t,t+\d)$.

Of course, there may be particles not in
 $\bigcup_{\f\ge 0} \l\{S^+(\f,\b)\cup S^-(\f,\b)\r\}$.

We define $N^+(\f,\b)$ [respectively, $N^-(\f,\b)$] to be the
number of particles in $S^+(\f,\b)$ [resp., $S^-(\f,\b)$].

We let $\oP^+_{\b}$ [respectively, $\oP^-_{\b}$] be the
probability law of our model when it starts with $1$ particle of
type $+$ [resp., $-$] at position $0$ at time $0$; and we let
$\oE^+_{\b}$ [resp., $\oE^-_{\b}$] be the associated expectation.

\emph{We often suppress the `$\b$' in the notation for
$\oP^{\pm}_{\b}$, $\oE^{\pm}_{\b}$, $S^{\pm}_{\b}$,
$N^{\pm}_{\b}$.}

Then, under $\oP^+$,  $N^+ = \{N^+(\f): \f\ge 0\}$ is a standard
branching process, in which a particle dies at rate $K_+$ and is
replaced at the `$\Phi$-time' of its death by a random
non-negative number, possibly $1$ and possibly $\infty$, of
children, the numbers of children being independent, identically
distributed random variables.  I take this as intuitively obvious,
and I am not going to ruin the paper by spelling out a proof.

Note that in the $\oP^-$ branching process $N^- = \{N^-(\f): \f\ge
0\}$, a particle may die without giving birth.

For $0\le \t < 1$, define
\[
   g^{++}(\f,\t) := \oE^+ \t^{N^+(\f)}, \quad
   h^{-+}(\f,\t) := \oE^- \t^{N^+(\f)}.
\]
Clearly, for $0\le \t < 1$,
\begin{align*}
   h^{-+}(\f,\t) &\Eq
   \oE^{-}\oE^{-}\bigl[\t^{N^+(\f)}\  \big|\ N^+(0)\bigr]
   \Eq \oE^- g^{++}(\f,\t)^{N^+(0)}
\\
   &\Eq H^{-+}\l(g^{++}(\f,\t)\r),
\end{align*}
where
\[
   H^{-+}(\t) = \oE^-\t^{N^+(0)}
   = \sum h^{-+}_n \t^n,
\]
where
\[
  h^{-+}_n := \oP^{-}[N^+(0)=n].
\]
It may well be that $h^{-+}_{\infty} := \oP^-[N^+(0)=\infty] >0$.
Note that
\[
    H^{-+}(1-) \Def \lim_{\t\uparrow 1}H^{-+}(\t)
    \Eq \oP^-[N^+(0)< \infty].
\]

\subsection{The dynamical system}
We now take $\t$ in $(0,1)$ and derive the backward differential
equations for
\[
   x(\f) := g^{++}(\f,\t),\quad y(\f) := h^{-+}(\f,\t),
\]
in the good old way in which we teach Applied Probability, and
then study the equations.  \cite{WNL} looks a bit more `rigorous'
here.

Consideration of what happens between times $0$ and $\rd t$ tells
us that
\[
  x(\f+\rd \f) = \{1 - K_+\,\rd \f\}x(\f) + \{q_+\,\rd \f\}y(\f) +
    \{\b\,\rd \f\} x(\f)^2 + \mathrm{o}(\rd\f).
\]
The point here is of course that if we started with $2$ particles
in the $+$ state, then $\oE\t^{N^+(\f)}=x(\f)^2$. We see that,
with $x'$ meaning $x'(\f)$,
\begin{subequations}
\label{DynSys}
\begin{equation}\label{DynSysa}
  x' = q_+(y - x) + \b(x^2 - x).
\end{equation}
Similarly, remembering that $\Phi$ starts to run backwards when
the particle starts in state $-$, we find that
\[
  y(\f - \rd\f) = \{1 - K_-\,\rd \f\}y(\f) +  \{q_-\,\rd \f\}x(\f) +
    \{\b\,\rd \f\} y(\f)^2 + \mathrm{o}(\rd\f),
\]
whence
\begin{equation}\label{DynSysb}
   -y' = q_-(x - y) + \b(y^2 - y).
\end{equation}
\end{subequations}
Of course, $y=H^{-+}(x)$ must represent the track of an integral
curve of the \emph{dynamical system} \eqref{DynSys}, and since $y'
= {H^{-+}}'(x)x'$, we have an \emph{autonomous equation for
$H^{-+}$} which we shall utilize below.

Note that the symmetry of the situation shows that $x = H^{+-}(y)$
must also represent the track of an integral curve of our
dynamical system, though one traversed in the `$\f$-reversed'
direction.

Probability Theory guarantees the existence of the `probabilistic
solutions' of the dynamical system tracking curves $y=H^{-+}(x)$
and $x=H^{+-}(y)$.

\begin{lemma}\label{NoEquil}
There can be no equilibria of our dynamical system in the interior
of the unit square.
\end{lemma}
\begin{proof}
For if $(x,y)$ is in the interior and
\[
  q_+(y-x)+\b(x^2-x) = 0, \quad q_-(x-y)+\b(y^2-y) = 0,
\]
then $y\ge x$ from the first equation and $x\ge y$ from the
second. Hence $x=y$ and $x^2-x = y^2 - y=0$.
\end{proof}

\subsection{Change of $\t$}
We need to think about how a change of $\t$ would affect things.
Suppose that $\a=\oE^+ \t^{N^+(\psi)}$ where $0<\a<1$.  Then
\[
   \oE^+ \oE^+\bigl[\t^{N^+(\f+\psi)}\ \big| \ N^+(\f)\bigr]
   \Eq
   \oE^+ \a^{N^+(\f)}
   \Eq
   g^{++}(\f,\a),
\]
where $\a = g^{++}(\psi,\t)$.  So, we have the
\emph{probabilistic-flow relation}
\beqn{probFlow}
  g^{++}(\f+\psi,\t) \Eq g^{++}\l(\f,g^{++}(\psi,\t)\r).
\ee
Likewise, $h^{-+}(\f+\psi,\t) \Eq h^{-+}\l(\f,g^{++}(\psi,\t)\r)$.
Thus, changing from $\t$ to $\a=\oE^+ \t^{N^+(\psi)}$ just changes
the starting-point of the motion along the probabilistic curve
from $(\t,H^{-+}(\t))$ to the point $(\a,H^{-+}(\a))$ still on the
probabilistic curve.  This is why we may sometimes seem not to
care about $\t$, and why it is not in our notation for $x(\f)$,
$y(\f)$.  But we shall discuss $\t$ when necessary, and the
extreme values $0$ and $1$ of $\t$ in Subsection \ref{01}.

If for \emph{any} starting point $\bv_0=(x_0,y_0)$ within the unit
square, we write $\bV(\f,\bv_0)$ for the value of $(x(\f),y(\f))$,
then, for values of $\f$ and $\psi$ in which we are interested, we
have (granted existence and uniqueness theorems) the
\emph{ODE-flow relation}
\[
   \bV(\f+\psi,\bv_0) \Eq \bV(\f, \bV(\psi,\bv_0))
\]
which generalizes \eqref{probFlow}. (The possibility of explosions
need not concern us: we are interested only in what happens within
the unit square.)  For background on ODE flows, see \cite{Arn}.

\section{How does ODE theory see the phase transition?}
\subsection{An existence theorem}
Even if you skip the (actually quite interesting!) proof of the
following theorem, do not skip the discussion of the result which
makes up the next subsection.

\begin{theorem}\label{Aexists}
There exist constants $\{a_n: n\ge 0\}$ with $a_0=0$, all other
$a_n$ strictly positive, and
\beqn{Ares}
    \sum a_n \le q_+/(q_- + \b),
\ee
and a solution $(x(\f),y(\f))$ of the `$\f$-reversed' dynamical
system
\[
   -x' =   q_+(y - x) + \b(x^2 - x),
   \qquad
   y' =    q_-(x - y) + \b(y^2 - y),
\]
such that $x(\f) = A(y(\f))$, where we now write $A(y) = \sum a_n
y^n$.
\end{theorem}

\begin{proof}
Assume that constants $a_n$ as described exist. Since $x'(\f) =
A'(y(\f))y'(\f)$, we have
\[
  - \l\{q_+y + \b A(y)^2 - K_+ A(y)\r\}
  = A'(y)\l\{q_- A(y) + \b y^2 - K_-y\r\}.
\]
Comparing coefficients of $y^0$,
\[
   \b a_0^2 - K_+a_0 = -a_1 q_- a_0,
\]
and we are guaranteeing this by taking $a_0=0$.  Comparing
coefficients of $y^1$, we obtain
\[
   q_- a_1^2 - (K_- + K_+)a_1 + q_+ = 0.
\]
We take
\begin{align*}
   a_1 &\Eq \frac{K_- + K_+ - \sqrt{(K_- + K_+)^2 - 4q_-q_+}}
       {2q_-}
\\
   &\Eq
   \frac{2q_+}
   {K_- + K_+ + \sqrt{(K_- + K_+)^2 - 4q_-q_+}}
\end{align*}
from which it is obvious that $0<a_1<1$.

On comparing coefficients of $y^n$ we find that, for $n\ge 2$,
\begin{align*}
  &\bigl\{K_+ + nK_{-}  - (n+1){q_-}a_1\bigr\}a_n\\
  &\Eq \b\sum_{k=1}^{n-1} a_k a_{n-k}
  + q_- \sum_{k=1}^{n-2}(k+1)a_{k+1}a_{n-k} + \b(n-1)a_{n-1}.
\end{align*}
We now consider the $a_n$ as being \emph{defined} by these
recurrence relations (and the values of $a_0$ and $a_1$). It is
clear that the $a_n$ are all positive.

Temporarily fix $N>2$, and define
\begin{align*}
   A_N(y) &\Def \sum_{n=0}^N a_n y^n,
\\
   L(y) &\Def \sum _{n=0}^N \ell_n y^n
   \Def -q_+y - \b A_N(y)^2 + K_+A_N(y),
\\
   R(y) & \Def \sum r_n y^n
   \Def A'_N(y)\bigl\{q_-A_N(y) + \b y^2 - K_-y\bigr\}.
\end{align*}
For $n\le N$ we have $\ell_n = r_n$ by the recurrence relations.
It is clear that for $n>N$ we have $\ell_n\le 0$ and $r_n\ge 0$.
Hence, for all $y$ in $(0,1)$,
\beqn{ANineq}
   -q_+y - \b A_N(y)^2 + K_+A_N(y)
   \Le
   A'_N(y)\bigl\{q_- A_N(y) + \b y^2 - K_-y\bigr\}.
\ee
Suppose for the purpose of contradiction that there exists $y_0$
in $(0,1)$ with $A_N(y_0)=y_0$. Then
\[
   -q_+ - \b y_0 + K_+ \le A'_N(y_0)\bigl\{q_- + \b y_0 - K_-\bigr\}.
\]
However, the left-hand side is positive while the right-hand side
is negative.

Because $A_N(0)=0$ and $A'_N(0) = a_1 < 1$, the contradiction
establishes that $A_N(y)<y$ for $y\in (0,1)$, so that $A_N(1)\le
1$. Since this is true for every $N$, and each $a_n$ ($n>1$) is
strictly positive, we have $A_N(1)<1$ for every $N$.

By inequality \eqref{ANineq}, we have
\[
   D_N q_- (A_N - 1) + q_+ + \b A_N^2 - K_+A_N \Ge 0,
\]
where $A_N:=A_N(1)<1$ and $D_N:= A_N'(1)$. Because each $a_n$
($n>0$) is positive it is clear that $A_N<D_N$.  We therefore have
\[
  q_+ + \b A_N^2 - K_+ A_N \ge D_N q_- (1-A_N) \ge q_- A_N(1-A_N),
\]
which simplifies to
\[
   (1-A_N)\{q_+ - (\b + q_-)A_N\}\ge 0.
\]
Since $(1-A_N)>0$, we have $(\b + q_-)A_N \le q_+$, and result
\eqref{Ares} follows.

It is clear that we now need to consider the autonomous equation
for $y=y(\f)$:
\[
  y' \Eq q_-[A(y)-y] + \b[y^2-y], \qquad y(0)=\t.
\]
But we can describe $y(\f)$ as $\oE \t^{Z(\f)}$ where
$\{Z(\f):\f\ge 0\}$ is a classical branching process in which
(with the usual independence properties) a particle dies at rate
$K_-$ and at the moment of its death gives birth to $C$ children
where
\[
    \oP(C=n) = \begin{cases}
    q_- a_n /K_- &\mbox{if $1\le n \le \infty$ and $n\ne 2$},\\
    (\b + q_- a_2)/K_- &\mbox{if $n=2$}.
    \end{cases}
\]
Of course, $a_{\infty} = 1 - \sum a_n$.

Then $(x(\f), y(\f)) = (A(y(\f)), y(\f))$ describes the desired
solution starting from $(A(\t),\t)$.
\end{proof}

\subsection{Important discussion}

Of course, ODE theory cannot see what we shall see later: namely
that $A(\cdot)=H^{+-}(\cdot)$ when $\b>\b_{\rm c}$ but
$A(\cdot)\ne H^{+-}(\cdot)$ when $\b\le\b_{\rm c}$.  When
$\b>\b_{\rm c}$, the curve $x = H^{+-}(y)$ is the steep bold curve
$x = A(y)$ at the \emph{left}-hand side of the picture as in
Figure \ref{Super}. But when $\b \le \b_{\rm c}$, the curve $x =
H^{+-}(y)$ is the steep bold curve at the \emph{right}-hand side
of the picture as in Figure \ref{crit}.  Ignore the shaded
triangle for now.

What ODE theory must see is that whereas there is only one
integral curve linking the top and bottom of the unit square when
$\b>\b_{\rm c}$, there are infinitely many such curves when
$\b\le\b_{\rm c}$ of which two, the curves $x = A(y)$ and $x =
H^{+-}(y)$, derive from probability generating functions (pgfs).

It does not seem at all easy to prove by Analysis that, when
$\b\le \b_{\rm c}$, there is an
integral curve linking the bottom
of the unit square to the point $(1,1)$, of the form $x = F(y)$
where $F$ is the pgf of a random variable which can perhaps take
the value $\infty$.  Methods such as that used to prove Theorem
\ref{Aexists} will not work.

Moreover, it is not easy to compute $H^{+-}(0)$ when $\b$ is equal
to, or close to, $\b_{\rm c}$.  If for example, $q_+=1$, $q_-=4$
and $\b=0.4$, then one can be certain that $H^{+-}(0)=0.6182$ to
$4$ places, and indeed one can easily calculate it to arbitrary
accuracy. But the critical nature of $\b_{\rm c}$ shows itself in
unstable behaviour of some na\"\i ve computer programs when $\b$
is equal to, or close to, $\b_{\rm c}$.  I believe that in the
critical case when $q_+=1$, $q_-=4$ and $\b=0.5$, $H^{+-}(0)$ is
just above $0.6290$.

\emph{Mathematica} is understandably extremely cautious in regard
to the non-linear dynamical system \eqref{DynSys}, and drives one
crazy with warnings.  If forced to produce pictures, it can
produce some rather crazy ones, though usually, but not absolutely
always, under protest.  Its pictures can be coaxed to agree with
those in the earlier appetizer which were produced from my own `C'
Runge--Kutta program which yielded Postscript output. Sadly, that
program and lots of others were lost in a computer burn-out before
I backed them up.

\section{Proof of Theorem \ref{phase} and more}
\begin{figure}[ht]
\begin{center}
\includegraphics{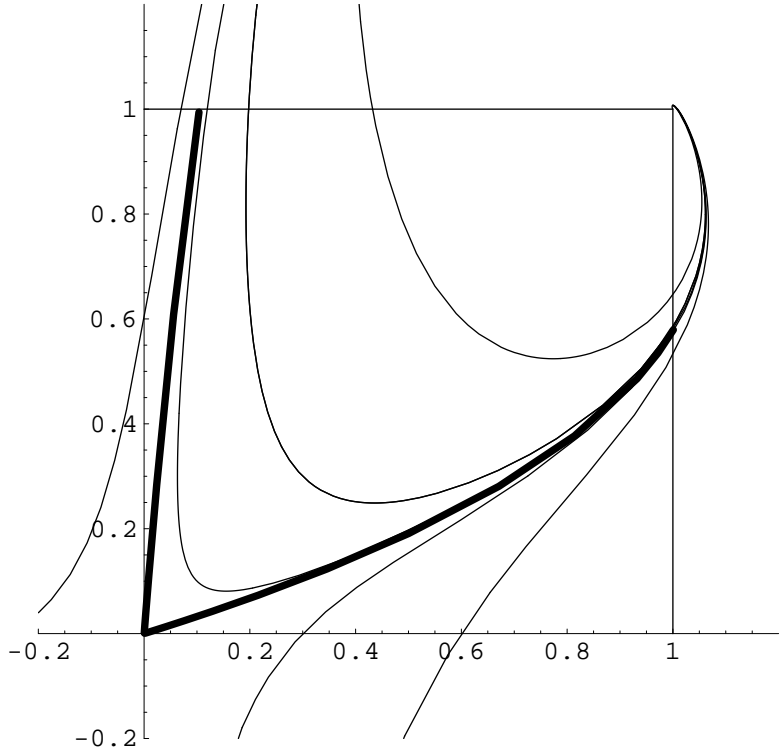}
\end{center}
\caption{\emph{Mathematica} picture of
the supercritical case when $q_+=1$, $q_-=4$, $\b = 4$}
\label{Super}
\end{figure}
\subsection{When $\b>\b_{\rm c}$}
\begin{lemma}
When $\b > \b_{\rm c}$,
\begin{enumerate}[{\rm (a)}]
\item $H^{-+}(0) = \oP^-[N^+(0) = 0] = 0$,
\item $H^{-+}(1-) = \oP^-[N^+(0) < \infty] < 1$,
\item $H^{+-}(1-) = \oP^+[N^-(0) < \infty] < 1$,
\item $H^{+-}(0) = \oP^+[N^-(0) = 0] = 0$.
\end{enumerate}
\end{lemma}
It is clearly enough to prove the lemma under the assumption
\[
        {\textstyle \frac12}\l(\sqrt{q_-}-\sqrt{q_+}\r)^2 < \b <
        {\textstyle \frac12}\l(\sqrt{q_-}+\sqrt{q_+}\r)^2,
\]
and this is made throughout the proof.
\begin{proof}
Result (a) is obvious.

The point $(1,1)$ is an equilibrium point of our dynamical system,
and we consider the linearization of the system near this
equilibrium.  We put $x=1+\xi$, $y=1 + \eta$ and linearize by
ignoring terms in $\xi^2$ and $\eta^2$:
\[
  \mat{\xi'\\ \eta'}
  \Eq
  \mat{-q_+ + \b & q_+ \\
        -q_-  &q_- - \b}
   \mat{\xi\\ \eta},
\]
the matrix being the linearization matrix of our system at
$(1,1)$. The characteristic equation for the eigenvalues of this
matrix is
\[
   \lm^2 + (q_+ - q_-)\lm + (q_- + q_+)\b - \b^2 = 0.
\]
The discriminant `$B^2 - 4AC$' is
\[
   \{2\b - (q_+ + q_-)\}^2 - 4q_-q_+.
\]
This expression is zero if $\b = \frac12(\sqrt{q_-} \pm
\sqrt{q_+})^2$. So, in our case, the eigenvalues $\lm$ have
non-zero imaginary parts.  Any solution of our system converging
to $(1,1)$ as $\f\to\pm\infty$ must spiral, and cannot remain
inside the unit square. Hence results (b) and (c) hold.

It is now topologically obvious (since there are no equilibria
inside the unit square) that we must have $\oP^+[N^-(0)=0]=0$;
otherwise how could the curve $x = H^{+-}(y)$ link the top and
bottom edges of the square?  Thus result (d) holds.
\end{proof}

Of course, we can now deduce from result \eqref{Ares} that (when
$\b>\b_{\rm c}$)
\[
   H^{+-}(1-) = \oP^+[N^-(0)<\infty] \le q_+/(q_- + \b).
\]

Figure \ref{Super}, which required `cooking' beyond choosing
different $\f$-ranges for different curves, represents the case
when $\ q_+=1$, $q_-=4$ and $\b=4$. The lower bold curve
represents $y = H^{-+}(x)$ and the upper $x = H^{+-}(y)$.  As
mentioned previously, $H^{+-}(y) = A(y)$ where $A$ is the function
of Theorem \ref{Aexists}.

The motion along the lower probabilistic curve $y=H^{-+}(x)$ will
start at $(\t,H^{-+}(\t))$ and move towards $(0,0)$ converging to
$(0,0)$ as $\f\to \infty$ since $N^{\pm}(\f)\to\infty$.  If we fix
$\f$ and let $\t\to 1$, we move along the curve towards the point
$(1,H(1-))$.  Of course, we could alternatively leave $\t$ fixed
and run $\f$ backwards.  It is clear because of the spiralling
around $(1,1)$ that the power series $H^{-+}(x)$ must have a
singularity at some point $x$ not far to the right of $1$.

Motion of the dynamical system along the steep probabilistic curve
$x=H^{+-}(y)$ on the left of the picture will be upwards because
it is the $\f$-reversal of the natural probabilistic motion. Now
you understand the sweep of the curves in the top-right of the
picture.

\subsection{A simple large-deviation result}\label{XPhi}
Let $\{X(t): t\ge 0\}$ be a Markov chain on $\{+,-\}$ with
$Q$-matrix
\[
    Q = \mat{-q_+&q_+\\
         q_-& -q_-
         }.
\]
Let $V$ be the function on $\{+,-\}$ with $V(+)=1$ and $V(-)=-1$
and define $\Phi_X(t) = \int_0^t V(X_s)\,\rd s$. Almost surely,
$\Phi_X(t)\to\infty$.  We stay in `dynamical-system mode' to
obtain the appropriate Feynman--Kac formula.

Let $\mu>0$, and define (with the obvious meanings of $\oE^{\pm}$)
\[
   u(t) = \oE^+\exp\{-\mu\Phi_X(t)\}, \quad
   v(t) = \oE^-\exp\{-\mu\Phi_X(t)\}.
\]
Then
\[
  u(t+\rd t) \Eq \{1 - q_+ \,\rd t\}e^{-\mu\, \rd t} u(t) + q_+\,\rd t\; v(t)
  + {\rm o}(\rd t),
\]
with a similar equation for $v$. We find that
\[
  \mat{u'\\ v'}
  = \mat{-q_+ - \mu & q_+ \\
  q_- & -q_- + \mu}
  \mat{u\\v},
  \quad
  \mat{u(t)\\v(t)}
  = \exp\{t(Q-\mu V)\}\mat{1\\1},
\]
where $V$ also denotes the operator
$\bigl(\begin{smallmatrix}1&0\\\noalign{\vskip2pt}0&-1\end{smallmatrix}\bigr)$
of multiplication by the function $V$.

\begin{lemma}\label{FKlem}
If $\b<\b_{\rm c} := \frac12\l(\sqrt{q_-} - \sqrt{q_+}\r)^2$, then
there exist positive constants $\e$, $\kp$, $A$ such that
\[
  \re^{\b t}\oP^{\pm}[\Phi_X(t)\le \e t] \le A\re^{-\kp t}.
\]
\end{lemma}
\begin{proof}
We have just shown that
\[
  \oE^{\pm}[\re^{-\mu\Phi_X(t)}f(X_t)] \Eq
  \exp\{t(Q-\mu V)\}f.
\]
Now $Q-\mu V$ has larger eigenvalue
\[
   \g \Eq -{\textstyle \frac12}(q_- + q_+)
   + {\textstyle \frac12}\sqrt{(q_- + q_+)^2 - 4(q_- - q_+)\mu + 4\mu^2}.
\]
We fix $\mu$ at $\frac12(q_- - q_+)$ to obtain the minimum value
$\frac12\l(\sqrt{q_-} - \sqrt{q_+}\r)^2$ of $\g$. Hence, for $\e>0$
and some constant $A_{\e}$,
\begin{align*}
  \oP^{\pm}\l[\Phi_X(t) \le \e t\r]
  &\Eq \oP^{\pm}\bigl[\mu(\e t - \Phi_X(t))\ge 0\bigr]
  \Le \oE^{\pm}\exp\{\mu\e t - \mu\Phi_X(t)\}
\\
  &\Le A_{\e}\exp\l\{{\textstyle \frac12}\e(q_- - q_+)t -
  {\textstyle \frac12}\l(\sqrt{q_-} - \sqrt{q_+}\r)^2 t\r\}.
\end{align*}
The lemma follows.
\end{proof}

For a fine paper proving very precise large-deviation results for
Markov chains via explicit calculation, see \citet*{BHK}.

\subsection{When $\b < \b_{\rm c}$}\label{SSsubcrit}
\begin{figure}[ht]
\begin{center}
\includegraphics[scale = 0.8]{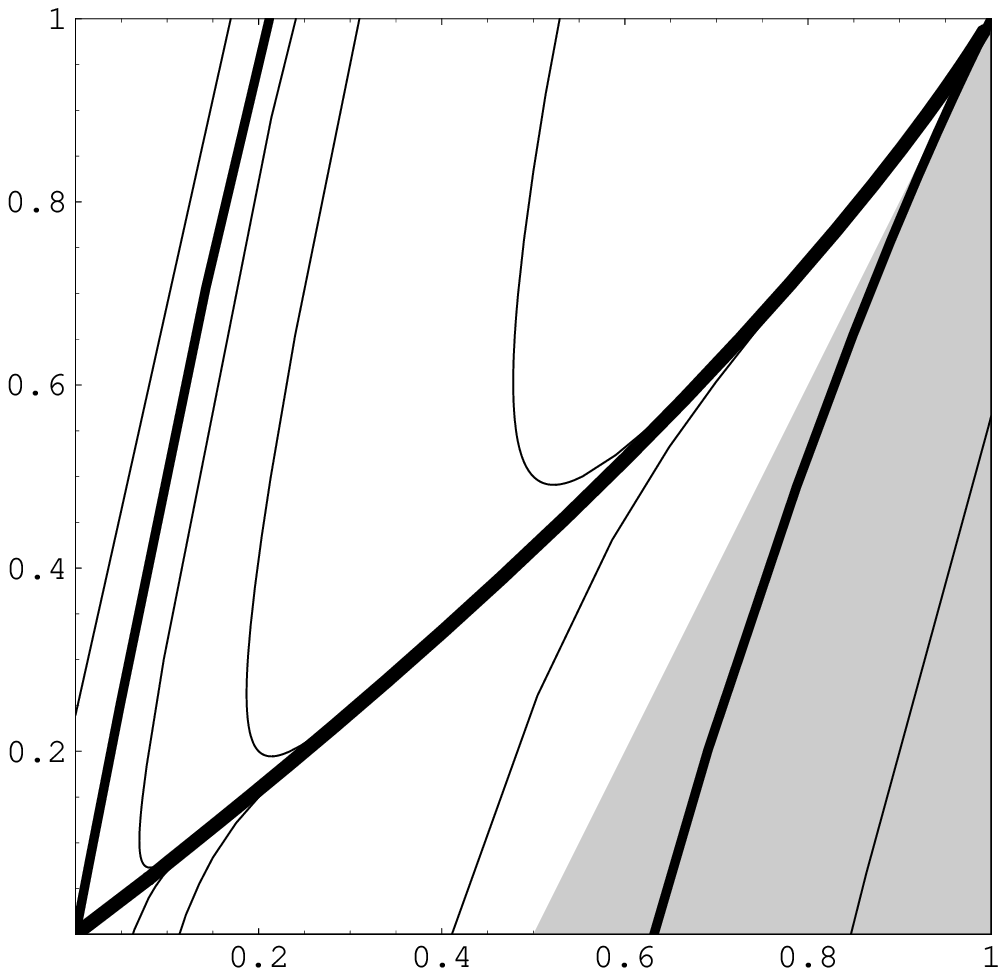}
\end{center}
\caption{\emph{Mathematica} picture of
the critical case when $q_+=1$, $q_-=4$, $\b = 0.5$}
\label{crit}
\end{figure}
\begin{lemma}\label{Lsubcrit}
When $\b\le \b_{\rm c}$,
\begin{enumerate}[{\rm (a)}]
\item $H^{-+}(0) = \oP^-[N^+(0) = 0] = 0$,
\item $H^{-+}(1-) = \oP^-[N^+(0) < \infty] = 1$,
\item $H^{+-}(0) = \oP^+[N^-(0) = 0] > 0$,
\item $H^{+-}(1-) = \oP^+[N^-(0)< \infty] = 1$.
\end{enumerate}
\end{lemma}

\begin{note} Though Figure \ref{crit} relates to the case
when $\b=\b_{\rm c}$, pictures for the subcritical case when
$\b<\b_{\rm c}$ look very much the same.
\end{note}

\begin{proof}
Result (a) remains trivial.

By Lemma \ref{FKlem}, there exist $\e>0$, $\kp>0$ and $A>0$ such
that, for a single particle moving according to $Q$-matrix $Q$, we
have
\[
   \re^{\b t}\oP[\Phi_X(t)\le \e t] \le A\re^{-\kp t}.
\]
For the branching process, the expression on the left-hand side is
the expected number of particles with $\Phi$-value less than or
equal to $\e t$ at real time $t$. So the probability that some
particle has $\Phi$-value less than or equal to $t$ is at most
$A\re^{-\kp t}$.

By the Borel--Cantelli Lemma, there will almost surely be a random
positive integer $n_0$ such that for all $n\ge n_0$, every
particle alive at real time $n$ will have $\Phi$-value greater
than $\e n$.  Since $\Phi$ can only move left at speed $1$, there
must almost surely come a time after which no particle has a
positive $\Phi$-value.  Hence $\oP^-[N^+(0)=\infty] = 0$, and
result (b) is proved.

Now suppose for the purpose of contradiction that $\oP^+[N^-(0)>0]
= 1$.  Since a particle started at state $+$ can remain there for
an arbitrary long time without giving birth, it follows that any
particle in the $+$ state and with any positive $\Phi$-value will
have a descendant for which $\Phi$ will become negative.  This
contradicts what we proved in the previous paragraph, so result
(c) is established.

Since the $y=H^{-+}(x)$ curve connects $(1,1)$ to $(0,0)$ and the
other probabilistic curve $x=H^{+-}(y)$ starts at $(H^{+-}(0),0)$
where $H^{+-}(0)>0$, and since these curves cannot cross at an
interior point of the unit square, it must be the case that
$H^{+-}(1-) = 1$, so that property (d) holds.
\end{proof}

In the analogue of Figure \ref{crit} for a subcritical case
(which, as I have said, looks very much like Figure \ref{crit}),
motion along the higher probabilistic curve $y=H^{-+}(x)$ will
again start at $(\t,H^{-+}(\t))$ and move towards $(0,0)$, this
because $N^+(\f)\to\infty$. Since $N^-(\f)\to 0$, the natural
probabilistic motion of the lower curve $x=H^{+-}(y)$ will
converge to $(1,1)$; but the $\f$-reversal means that the
dynamical system will move downwards along this curve.

\begin{proof}[Sketch of geometric proof that $H^{-+}(1-) = 1$ if
$\b\le \b_{\rm c}$] It is enough to prove the result when
$\b=\b_{\rm c}$.  Let $m = \sqrt{q_-/q_+}$, the slope of the
unique eigenvector of the linearity matrix at $(1,1)$. Draw the
line of slope $m$ from $(1,1)$ down to the $y$-axis, the sloping
side of the shaded triangle in the picture.  Now it is
particularly easy to check that at any point of the sloping side
the ($\rd y/\rd x$)-slope of an integral curve is greater than
$m$. If the convex curve $y=H^{-+}(x)$ hit the vertical side of
the triangle at any point lower than $1$, we would have
`contradiction of slopes' where it crossed that sloping side.
\end{proof}

\subsection{Nested models and continuity at phase
transition}
Take $\b_0>\b_{\rm c}$ and let $M_{\b_0}$ be our model
with initial law $\oP^+_{\b_0}$ (in the obvious sense).  Label
birth-times $T_1$, $T_2$, $T_3$, \ldots\ in the order in which
they occur, and for each $n$ call one of the two children born at
$T_n$ `first', the other `second'. Let $U_1$, $U_2$, $U_3$,
\ldots\ be independent random variables each with the uniform
distribution on $[0,1)$. We construct a nested family of models
$\{M_{\b}:\b\le \b_0\}$ as follows.

Fix $\b<\b_0$ for the moment. If $U_n>\b/\b_0$, erase the whole
family tree in $M_{\b_0}$ descended from the second child born at
time $T_n$. Of course, this family tree may already have been
erased earlier. In this way, we have a model $M_{\b}$ with desired
law $\oP^+_{\b}$.  The set $S^-(\f,\b)$ will now denote the
$S^-(\f)$ set for the `nested' model $M_{\b}$, and $N^-(\f,\b)$
will denote its cardinality.

A particle $i$ contributing to $S^-(0,\b)$ determines a path in
$\{+,-\}\times [0,\infty)$:
\[
    \{(\mbox{Ancestor}_i(t), \Phi_i(t)): t < \rho_i\}
\]
where $\rho_i$ is the first time after which $\Phi_i$ becomes
negative.  Along that $M_{\b}$-path, there will be finitely many
births. Now, for fixed $\b$ it is almost surely true that $U_n\ne
\b$ for all $n$. It is therefore clear that, almost surely, for
$\b'<\b$ and $\b'$ sufficiently close to $\b$, the $M_{\b}$-path
will also be a path of $M_{\b'}$.  In other words, we have the
left-continuity property
\[
     S^-(0,\b) = \bigcup_{\b'<\b}S^-(0,\b'),\mbox{ almost surely}.
\]
It therefore follows from the Monotone-Convergence Theorem that
\beqn{leftcont}
    \oE^+ N^-(0,\b_{\rm c}) = \uparrow \lim_{\b \uparrow \b_{\rm c}}
    \oE^+ N^-(0,\b).
\ee

Clearly, something goes seriously wrong in regard to
right-continuity at $\b_{\rm c}$. Suppose we have a path which
contributes to $S^-(0,\b)$ for all $\b>\b_{\rm c}$.  Then, for all
birth-times $T_n$ along that path we have $U_n\le \b/\b_0$ for all
$\b>\b_{\rm c}$ and hence $U_n\le \b_{\rm c}/\b_0$.  Hence
\[
   S^-(0,\b_{\rm c}) = \bigcap_{\b>\b_{\rm c}}S^-(0,\b).
\]
But it is clearly possible to have a decreasing sequence of
infinite sets with finite intersection.  And recall that (more
generally) the Monotone-Convergence Theorem is guaranteed to work
`downwards' (via the Dom\-inated-Convergence Theorem) only when one
of the random variables has finite expectation.

\subsection{Expectations and an embedded discrete-parameter
branching process}
If either of the curves $y=H^{-+}(x)$ or $x = H^{+-}(y)$
approaches $(1,1)$, it must do so in a definite direction and it
is well known (and an immediate consequence of l'H\^opital's rule)
that that direction must be an eigenvector of the linearity matrix
at $(1,1)$. When $\b=\b_{\rm c}$, there is (as we have seen
before) only one eigenvector
$\bigl(\begin{smallmatrix}m\\\noalign{\vskip2pt}1\end{smallmatrix}\bigr)$
with $m =\sqrt{q_-/q_+}$.  Thus
\beqn{crit-+}
   \oE^- N^+(0,\b_{\rm c}) = (H^{-+})'(1,\b_{\rm c}) =
   m = \sqrt{q_-/q_+}.
\ee

We know that if $\b< \b_{\rm c}$, then $H^{+-}(1-) = 1$ and we can
easily check that (as geometry would lead us to guess)
\[
  \oE^+ N^-(0,\b) = (H^{+-})'(1,\b) \le  1/m = \sqrt{q_+/q_-},
\]
and now, by equation \eqref{leftcont} we see that
\beqn{crit+-}
  \oE^+ N^-(0,\b_{\rm c}) \le  1/m = \sqrt{q_+/q_-}.
\ee
In particular, $\oP^+[N^-(0,\b_{\rm c})=\infty] = 0$, and so, in
fact, $H^{+-}(1)=1$ and we have equality at \eqref{crit+-}, whence
\beqn{prob>0}
   \oP^+[N^-(0,\b_{\rm c})\ge 1] \le
   \oE^+ N^-(0,\b_{\rm c}) = \sqrt{q_-/q_+},
\ee
part of Theorem \ref{phase}.

Now, for $\f>0$, let
\[
   b(\f) = \oE^+ N^+(\f), \quad c(\f) = \oE^- N^+(\f).
\]
Then
\[
   b'(\f) = q_+\{c(\f)-b(\f)\} + \b b(\f), \quad \mbox{etc.,}
\]
so that the linearization matrix at $(1,1)$ controls expectations.
We easily deduce the following theorem.

\begin{theorem}\label{balance}
When $\b=\b_{\rm c}$,
\[\begin{array}{ll}
  \oE^+ N^+(\f)=\re^{\uhalf(q_- - q_+)\f},\quad
  &\oE^- N^+(\f)=\sqrt{\frac{q_-}{q_+}}\;
  \re^{\uhalf(q_- - q_+)\f},
\\
  \oE^+ N^-(\f)=\sqrt{\frac{q_+}{q_-}}\;
  \re^{-\uhalf(q_- - q_+)\f},\quad
  &\oE^- N^-(\f)=\re^{-\uhalf(q_- - q_+)\f}.
\end{array}
\]
\end{theorem}

For any $\b$, we can define the discrete-parameter
branching-processes $\{W^{\pm}(n):n\ge 0\}$ as follows.  Let $B_i$
be the birth time and $D_i$ the death time of particle~$i$. Recall
that $\Phi_i$ is defined on $[0,D_i)$.  Let $\s_i(0)=0$ and, for
$n\ge 1$, define
\[
  \s_i(n) \Def \inf\{t: B_i\le t < D_i: t>\s_i(n-1);
  (-1)^n\Phi_i(t)>0\},
\]
with the usual convention that the infimum of the empty set is
$\infty$.  Let
\[
   W^+(n) \Def \sharp\{i: \s_i(2n)<\infty\}, \quad
   W^-(n) \Def \sharp\{i: \s_i(2n+1)<\infty\}.
\]
The `$W$' notation is suggested by `winding operators' in linear
Wiener--Hopf theory.
\begin{theorem}\label{DPbranch}
$W^{\pm}$ is a classical branching process under $\oP^{+}_{\b}$,
and is critical when $\b=\b_{\rm c}$.
\end{theorem}
The proof (left to the reader) obviously hinges on the case when
$\f=0$ of Theorem \ref{balance}. And do have a think about the
consequences of the `balance'
\[
  \oE^- N^+(\f,\b_{\rm c})\,\oE^+ N^-(\f,\b_{\rm c}) \Eq 1
\]
in that theorem.

\subsection{When $\t=0$ or $1$}\label{01}
If we take $\t=0$ and set
\[
  b(\f) := \oP^+[N^-(\f)=0], \quad
  c(\f) := \oP^-[N^-(\f)=0],
\]
then $\{(b(\f),c(\f)): \f\ge 0\}$ is a solution of the
$\f$-reversed dynamical system such that
\[
   b(\f) = H^{+-}(c(\f)).
\]
When $\b>\b_{\rm c}$, this solution stays at equilibrium point
$(0,0)$. When $\b\le \b_{\rm c}$, $(b(\f),c(\f))$ moves (as $0\le
\f \uparrow\infty$) from $(H^{+-}(0),0)$ to $(1,1)$ tracing out
the right-hand bold curve in the appropriate version of Figure
\ref{crit}.
\bigskip

When $\t=1$, $\{(B(\f),C(\f)): \f\ge 0\}$, where
\[
  B(\f) := \oP^+[N^-(\f)<\infty], \quad
  C(\f) := \oP^-[N^-(\f)<\infty],
\]
gives a solution of the $\f$-reversed dynamical system.  When
$\b\le \b_{\rm c}$, this solution stays at the equilibrium point
$(1,1)$. When $\b>\b_{\rm c}$, $(B(\f),C(\f))$ moves (as $0< \f
\uparrow\infty$) from $(H^{+-}(1-),1)$ to $(0,0)$ tracing out the
bold upper curve in the appropriate version of Figure \ref{Super}.
Of course, there is an appropriate version (`with $+$ and $-$
interchanged') for the lower curve.

\subsection{A tantalizing question}
When $\b>\b_{\rm c}$, we have for the function
$A(\cdot)=A(\cdot,\b)$ of Theorem \ref{Aexists}:
\[
   A(\t,\b) = \oE^+\t^{N^-(0,\b)}.
\]

When $\b\le \b_{\rm c}$, we have, for $0<\t<1$, $A(\t,\b) =
\oE^+\t^{Y_{\b}}$ for some random variable $Y_{\b}$.  Can we find
such a $Y_{\b}$ which is naturally related to our model?  In
particular, can we do this when $\b=\b_{\rm c}$?  It would be very
illuminating if we could.

\emph{What is true for all $\b>0$} is that if $X$ and $\Phi_X$ are
as at the start of Subsection \ref{XPhi} and $\tau_X^-(0) :=
\inf\{t: \Phi_X(t)<0\}$, then
\[
  \oE^+ \exp\{-\b \tau_X^-(0)\} = a_1,
\]
with $a_1$ as in Theorem \ref{Aexists}; and this tallies with $a_1
= \oP^+[N^-(0)=1]$ when $\b>\b_{\rm c}$.  Proof of the statements
in this paragraph is left as an exercise.

\paragraph{Acknowledgements}
I certainly must thank a referee for pointing out some typos and,
more importantly, a piece of craziness in my original version of a
key definition. The referee also wished to draw our attention
(mine and yours) to an important survey article on tree-indexed
processes \cite{Pem} by Robin Pemantle. The process we have
studied is, of course, built on the tree associated with the
underlying branching process. But our Wiener--Hopferization makes
for a rather unorthodox problem.

I repeat my thanks to Chris Rogers and John Toland for help with
the previous appetizer.  My thanks to Ian Davies and Ben
Farrington for helpful discussions.

And I must thank the Henschel String Quartet --- and not only for
wonderful performances of Beethoven, Haydn, Mendelssohn, Schulhoff
and Ravel, and fascinating discussions on music and mathematics
generally and on the connections between the Mendelssohns and
Dirichlet, Kummer, Hensel, Hayman. If I hadn't witnessed something
of the quartet's astonishing dedication to music, it is very
likely that I would have been content to leave things with that
early appetizer and not made even the small advance which this
paper represents. So, Happy music making, Henschels!



\begin{thebibliography}{99}
  \expandafter\ifx\csname natexlab\endcsname\relax
    \def\natexlab#1{#1}\fi
  \expandafter\ifx\csname selectlanguage\endcsname\relax
    \def\selectlanguage#1{\relax}\fi


\bibitem[Arnol'd(2006)]{Arn}
Arnol'd, V. I. 2006. \emph{Ordinary Differential Equations}, transl. from
Russian by R. Cooke. Universitext. Berlin: Springer-Verlag.

\bibitem[Brydges et al.(2007)Brydges, van der Hofstad, and K\"onig]{BHK}
Brydges, D., Hofstad, R. van der, and K\"onig, W. 2007. Joint
density for the local times of continuous-time Markov chains.
\emph{Ann.\ Probab.,}\ \textbf{35}, 1307--1332.

\bibitem[Pemantle(1995)]{Pem}
Pemantle, R. 1995. Tree-indexed processes. \emph{Statist.\ Sci.},
\textbf{5}, 200--213.

\bibitem[Warren et al.(2000), Warren and Williams]{WW} Warren, J., and Williams, D. 2000.
Probabilistic study of a dynamical system. \emph{Proc.\ Lond.\ Math.\ Soc.\
(3),} \textbf{81}(3), 618--650.

\bibitem[Williams(1995)]{WNL} Williams, D. 1995. Non-linear
Wiener--Hopf theory, 1: an appetizer. Pages 155--161 of: Azema,
J., \'Emery, M., Meyer, P.-A., Yor, M.\ (eds), \emph{S\'eminaire
de probabilit\'es} \textbf{29}. Lecture Notes in Math.
\textbf{1613}. New York: Springer-Verlag.

\end{thebibliography}
\end{document}